\documentclass{amsproc}
\usepackage{graphicx,mathrsfs,amssymb,amsmath}

\newtheorem{theorem}{Theorem}[section]
\newtheorem{lemma}[theorem]{Lemma}
\newtheorem{corollary}[theorem]{Corollary}
\newtheorem{proposition}[theorem]{Proposition}
\newtheorem{remark}[theorem]{Remark}
\newtheorem{definition}[theorem]{Definition}
\newtheorem{example}[theorem]{Example}

\numberwithin{equation}{section}

\newcommand{\cz}{{\mathbb C}}

\newcommand{\nz}{{\mathbb N}}
\newcommand{\rz}{{\mathbb R}}
\newcommand{\sz}{{\mathbb S}}

\newcommand{\bfn}{\mathbf{n}}
\newcommand{\bfL}{\mathbf{L}}
\newcommand{\bfS}{\mathbf{S}}
\newcommand{\bfT}{\mathbf{T}}

\newcommand{\frakg}{\mathfrak{g}}
\newcommand{\frakG}{\mathfrak{G}}

\newcommand{\scrC}{\mathscr{C}}
\newcommand{\scrL}{\mathscr{L}}
\newcommand{\scrS}{\mathscr{S}}

\newcommand{\cl}{\mathrm{cl}}
\newcommand{\dbar}{d\hspace*{-0.08em}\bar{}\hspace*{0.1em}}
\newcommand{\eps}{\varepsilon}
\newcommand{\forget}[1]{}

\newcommand{\lra}{\longrightarrow}
\newcommand{\op}{\mathrm{op}}

\newcommand{\rpbar}{\overline{\rz}_+}
\newcommand{\smsum}{\mathop{\mbox{\large$\sum$}}}
\newcommand{\spk}[1]{\langle#1\rangle}

\newcommand{\wcl}{\text{$\mathrm{w}$-$\mathrm{cl}$}}
\newcommand{\wh}{\widehat}
\newcommand{\whsz}{\widehat{\mathbb S}}
\newcommand{\wt}{\widetilde}

\begin{document}
\title[Parameter-dependent Operators of Toeplitz Type]{Parameter-dependent Pseudodifferential\\ 
Operators of Toeplitz Type}

\author{J\"org Seiler}
\address{Universit\`{a} di Torino, Dipartimento di Matematica, 10123 Torino $($Italy$)$}
\email{joerg.seiler@unito.it}

\begin{abstract}
We present a calculus of pseudodifferential operators that contains both usual parameter-dependent operators -- 
where a real parameter $\tau$ enters as an additional covariable -- as well as operators not depending on $\tau$. 
Parameter-ellipticity is characterized by the invertibility of three associated principal symbols. 
The homogeneous principal symbol is not smooth on the whole co-sphere bundle but only admits directional limits 
at the north-poles, encoded by a principal angular symbol. Furthermore there is a limit-family for $\tau\to+\infty$. 
Ellipticity permits to construct parametrices that are inverses for large values of the parameter. 
We then obtain sub-calculi of Toeplitz type with a corresponding symbol structure. In particular, we discuss  
invertibility of operators of the form $P_1A(\tau)P_0$ where both $P_0$ and $P_1$ are zero-order projections 
and $A(\tau)$ is a usual parameter-dependent operator of arbitrary order or $A(\tau)=\tau^{\mu}-A$ with a 
pseudodifferential operator $A$ of positive integer order $\mu$. 
\end{abstract}

\maketitle

\section{Introduction}\label{sec:intro}

The theory of parameter-dependent pseudodifferential operators provides a systematic operator-algebraic approach 
to analyzing spectral properties of differential and pseudodifferential operators, in particular, the existence and the structure 
of their resolvents. After the classical works of Agmon \cite{Agmo} and Agranovich, Vishik \cite{AgVi} dealing with 
parameter-elliptic and parabolic boundary value problems, the foundations of this concept 
were layed in the classical works of Seeley \cite{Seel0}, \cite{Seel1}, \cite{Seel2} on complex powers of operators on smooth 
manifolds, also with boundary. From there on parameter-dependent pseudodifferential operators became a standard tool  
in partial differential equations as well as in geometric and global analysis. 

In \cite{SSS98} Schulze, Shatalov and Sternin developed a systematic theory of elliptic boundary value problems with 
boundary conditions of generalized Atiyah-Patodi-Singer type (with ellipticity being equivalent to the Fredholm 
property in associated Sobolev spaces). In difference to the usual theory of boundary value problems with 
Lopatinskii-Shapiro ellipticity, roughly speaking, the range spaces of the boundary conditions in this generalized set-up are 
not the full Sobolev spaces, but closed subspaces determined by pseudodifferential projections. In \cite{Schu37} Schulze 
embedded their results in the framework of a pseudodifferential calculus, an algebra of Boutet de Monvel type. Due to 
the presence of projections he introduced the notion of pseudodifferential operators of \emph{Toeplitz type}. 
In \cite{Seil} the author showed that elliptic theory for operators of Toeplitz type can be formulated on a quite general level, 
allowing to realize this concept for any ``reasonably nice" calculus of pseudodifferential operators (including, for example, 
operators on manifolds with conical singularities, and Boutet de Monvel's algebra for boundary value problems). 
It is also indicated how parameter-dependent operators can be treated on a general level. 
It is the aim of the present paper to realize the abstract approach of parameter-ellipticity for pseudodifferential operators 
of Toeplitz type concretely in the case of operators on Euklidean space and closed smooth manifolds. 

To illustrate in more detail the contents of this paper let us first recall the probably most elementary calculus of 
parameter-dependent operators, designed for describing resolvents of differential operators on closed manifolds: 
If $M$ is a closed manifold and $F_0,F_1$ are two vector bundles over $M$ let us denote by 
$L^\mu_\cl(\rpbar;M,F_0,F_1)$ the space of classical pseudodifferential operators of order $\mu\in\rz$ acting from 
sections of $F_0$ to sections of $F_1$ and that depend on a parameter $\tau\ge0$ that enters as an additional 
co-variable $($for simplicity of presentation we focus in this introduction on the case of a real parameter $\tau$, while 
later we shall admit a more general parameter space$)$. This means that in local coordinates $x$, with 
corresponding co-variable $\xi$, the local pseudodifferential symbols satisfy estimates of the form 
\begin{equation*}
 |D^\alpha_\xi D^\beta_x D^k_\tau a(x,\xi,\tau)|\le 
  C_{\alpha\beta\gamma}(1+|\xi|+|\tau|)^{\mu-|\alpha|-k}, 
\end{equation*}
Thus differentiation with respect to $\xi$ or $\tau$ improves the decay in $\xi$ and 
$\tau$ simultaneously. The phrase \emph{classical}, indicated by the subscript $\cl$, for us means that, additionally, 
$a$ has a complete asymptotic expansion into components that are positively homogeneous with respect to 
$(\xi,\tau)$; see Section \ref{sec:1} for more details. With any such operator one can associate a homogeneous 
principal symbol which is a bundle homomorphism 
 $$\pi^*F_0\lra\pi^*F_1,\qquad \pi:(T^*M\times\rpbar)\setminus\{0\}\lra M,$$
where $\pi$ is the canonical projection onto $M$. Due to the homogeneity this homomorphism is uniquely determined 
by its restriction to the unit-cosphere bundle. 
An operator $A(\tau)$ is called \emph{parameter-elliptic} if its 
homogeneous principal symbol is an isomorphism. In this case one can find a parametrix $B(\tau)$ of corresponding  
negativ order that coincides with the inverse $A(\tau)^{-1}$ for sufficiently large values of $\tau$. 
For example, $\tau^2-\Delta$ with the Laplacian on $M$ is a parameter-elliptic operator of order $2$, and its 
parametrix/inverse is of order $-2$. 

Note that one cannot proceed in this way when the Laplacian is replaced by a 
pseudodifferential $($and non-differential$)$ operator $A$. This is due to the fact that 
$\tau^{\mu}-A$, $\mu=\mathrm{ord}\,A$, is not a parameter-dependent operator in the above described sense. 
To cover this case Grubb in \cite{Grub} introduced a more general calculus of parameter-dependent operators 
$($actually in her book a parameter-dependent version of Boutet de Monvel's algebra for boundary value problems is 
developed, containing operators on a closed manifold as a simple special case$)$ that she later in \cite{GrSe} together 
with Seeley further refined for studying resolvent trace asymptotics for pseudodifferential operators and certain 
non-local boundary value problems.

The main motivation for the present paper was to study the invertibility of operators of the form  
$S(\tau)=P_1A(\tau)P_0$ for large values of $\tau$ and the structure of the inverse, where 
$A(\tau)\in L^\mu_\cl(\rpbar;M,F_0,F_1)$ and $P_j\in L^0_\cl(M,F_j,F_j)$ are two zero-order pseudodifferential 
projections not depending on the parameter $($invertibility now refers to invertibility as a map between 
the closed subspaces determined by the projections, say the subspaces of the smooth sections of $F_0$ and $F_1$, 
respectively$)$. The difficulty lies in the fact that $L^0_\cl(M,F_j,F_j)$ is not a subset of $L^0_\cl(\rpbar;M,F_j,F_j)$;  
this prevents us from a direct application of the general sceme from \cite{Seil}. Also  the approach of both \cite{Grub} and 
\cite{GrSe} does not seem to apply to this situation. To cope with operators of the form $S(\tau)$ we shall construct, 
in Section \ref{sec:3}, a calculus of pseudodifferential operators containing both parameter-dependent operators in 
the above described sense of non-positive integer order as well as operators independent of the parameter of arbitrary 
order. In local coordinates, the resulting symbols have a complete asymptotic expansion into homogeneous components, 
but in contrast to the previously described standard set-up the components do not necessarily restrict to smooth 
functions on the whole $(\xi,\tau)$-unit sphere but have a specific singular structure at the ``north-pole" $(\xi,\tau)=(0,1)$: 
introducing polar-coordinates they admit Taylor-asymptotics; see Definition \ref{def:taylor1} for details. Globally, the 
homogeneous principal symbol is defined on the unit-cosphere bundle in $T^*M\times\rpbar$ with the section of 
north-poles removed. The parameter-ellipticity in this calculus is characterized by the invertibility of the homogeneous 
principal symbol outside the north-poles, the invertibility of the first Taylor term (globally a symbol on the cosphere 
bundle of $M$, which we call the principal \emph{angular symbol}$)$, and the invertibility of a certain \emph{limit-family} 
obtained as $\tau\to\infty$. 
Parameter-elliptic operators possess a parametrix within the calculus that coincides with the inverse for large values 
of the parameter. Once established this calculus we can employ the sceme of \cite{Seil} and derive the corresponding 
notion of parameter-ellipticity for parameter-dependent Toeplitz operators, in particular, for operators of the form 
$S(\tau)$. As it turns out also operators of the form $P_1(\tau^{\mu}-A)P_0$ with $A\in L^\mu_\cl(M,F,F)$ of positive 
integer order and a projections $P_j\in L^0_\cl(M,F,F)$ are covered. 

The calculus constructed in Section \ref{sec:3} is closely related to one earlier developed by 
Savin and Sternin in \cite{SaSt}. In this paper the 
authors introduce a calculus of zero-order operators on a fibration with both base and fibre being closed 
manifolds $($which is needed to describe a class of boundary value problems on manifolds with fibred boundary they are 
ultimately interested in$)$. This calculus comes along with a principal symbolic structure that corresponds to ours, 
i.e., homogeneous principal symbol, principal angular symbol, and limit-family. While we are interested in invertibility 
of operator-families for large values of the parameter, they also quantize the parameter (i.e., treat it as a covariable 
of the base manifold) and analyze the Fredholm property of the resulting operators. The residual class of their calculus 
consists of the compact operators, while in our calculus we have operators of any order and a complete 
asymptotic expansion into homogeneous components. 

\section{Important notation}\label{sec:1}

In this section we shall introduce some notation that will be used throughout the paper. 
With some fixed choice of $0\le a\le b<2\pi$ we set 
\begin{equation*}
 \Lambda=\Lambda(a,b):=\Big\{\lambda=(\tau,\theta) \;\big\vert\; \tau\ge 0,\; a\le\theta\le b\Big\}\subset\rz^2
\end{equation*}
$($in applications $\Lambda$ may also be identified, via polar-coordinates, with a sector in the complex plane$)$.  

\subsection{Pseudodifferential symbols}\label{sec:1.1}

With $\mu\in\rz$ and $N_0,N_1\in\nz$ we let $S^\mu(\rz^n,\rz^n\times\Lambda;N_0,N_1)$ denote the space of all 
functions $a:\rz^n_x\times\rz^n_\xi\times\Lambda\to\cz^{N_1\times N_0}$ taking values in the complex 
$(N_1\times N_0)$-matrices that are infinitely many times continuously differentiable with 
respect to $(x,\xi,\tau)$ and satisfy uniform estimates 
\begin{equation*}
 |D^\alpha_\xi D^\beta_x D^k_\tau a(x,\xi,\lambda)|\le 
  C_{\alpha\beta\gamma}\spk{\xi,\tau}^{\mu-|\alpha|-|\gamma|}
\end{equation*}
for any order of derivatives; here $\spk{\xi,\tau}:=(1+|\xi|^2+|\tau|^2)^{1/2}$. 
The symbol $a$ is called 
\emph{classical} if there exists a sequence of symbols $a_{\mu-j}\in  S^{\mu-j}(\rz^n,\rz^n\times\Lambda;N_0,N_1)$ 
which are homogeneous in the large of degree $\mu-j$, i.e., 
 $$a_{\mu-j}(x,t\xi,t\tau,\theta)=t^{\mu-j}a(x,\xi,\tau,\theta)\qquad \forall\;t\ge1\quad\forall\;|(\xi,\tau)|\ge1,$$
such that 
 $$ a-\smsum\limits_{j=0}^{\ell-1} a_{\mu-j}\;\in\; S^{\mu-j-\ell}(\rz^n,\rz^n\times\Lambda;N_0,N_1).$$
The functions 
\begin{equation*}
 a^{(\mu-j)}(x,\xi,\tau,\theta)=a_{\mu-j}\Big(x,\frac{(\xi,\tau)}{|(\xi,\tau)|},\theta\Big),
 \qquad (\xi,\tau)\not=0,
\end{equation*}
 are called the homogeneous components of $a$ and $a^{(\mu)}$ the \emph{homogeneous principal symbol} of $a$. 
The space of classical symbols of order $\mu$ shall be denoted by 
$S^{\mu}_\cl(\rz^n,\rz^n\times\Lambda;N_0,N_1)$.

We shall also use symbols without parameter. The classes $S^{\mu}(\rz^n,\rz^n;N_0,N_1)$ and 
$S^{\mu}_\cl(\rz^n,\rz^n;N_0,N_1)$ are defined as above by eliminating everywhere the parameter $\lambda$.  

With $a\in S^{\mu}(\rz^n,\rz^n;N_0,N_1)$ associate the operator $\op(a)=a(x,D)$ defined by 
 $$[\op(a)u](x)=\int e^{ix\xi}a(x,\xi)\wh{u}(\xi)\,\dbar\xi;$$
then $\op(a)$ defines a map $\scrS(\rz^n,\cz^{N_0})\to \scrS(\rz^n,\cz^{N_1})$ between the spaces of 
rapidly decreasing functions that extends by continuity to a map between the standard Sobolev (Bessel potential) 
spaces, $\op(a):H^s(\rz^n,\cz^{N_0})\to H^{s-\mu}(\rz^n,\cz^{N_1})$ for arbitrary $s\in\rz$.  

For convenience we shall frequently use the short-hand notations 
 $$S^\mu,\quad S^\mu_\cl,\quad S^\mu(\Lambda),\quad S^\mu_\cl(\Lambda);$$
in particular, the numbers $N_0$ and $N_1$ will be indicated only if necessary. 

Occasionally we shall also make use of a version of the symbol class $S^\mu$ where the symbols do not take values 
in a space of matrices but in a \emph{Fr\'{e}chet space} $E$; we shall denote this space by $S^\mu(E)$. A function 
$a:\rz^n\times\rz^n\to E$ belongs to $S^\mu(E)$ if it is smooth and satisfies uniform estimates 
\begin{equation*}\label{eq:frechet}
 p\big(D^\alpha_\xi D^\beta_x a(x,\xi)\big)\le C_{\alpha\beta p}(a)\spk{\xi}^{\mu-|\alpha|}
\end{equation*}
for any continuous semi-norm $p$ on $E$ and any order of derivatives. For example, we can consider 
$a(x,\xi)\in S^\mu(\rz^n,\rz^n;N_0,N_1)$ as a symbol $a(\xi)\in S^\mu(E)$ with 
$E=\scrC^\infty_b(\rz^n_x,\cz^{N_1\times N_0})$, the space of all smooth $\cz^{N_1\times N_0}$-valued 
functions having bounded derivatives of any order. 

\subsection{Pseudodifferential operators on closed manifolds}\label{sec:1.2}

We let $M$ denote a smooth closed Riemannian manifold of dimension $\mathrm{dim}\,M=n$. Using a partition of 
unity and local coordinates and local bundle trivialisations one can define the spaces  
 $$L^\mu_{(\cl)}(M,F_0,F_1), \qquad L^\mu_{(\cl)}(\Lambda;M,F_0,F_1)$$
of $($parameter-dependent$)$ pseudodifferential operators acting between sections of the hermitian 
vector-bundles $F_0$ and $F_1$. The local operators have symbols as described in the previous subsection with 
$N_j=\mathrm{dim}\,F_j$ for $j=0,1$, while the spaces of global smoothing operators 
 $$L^{-\infty}(M,F_0,F_1),\qquad L^{-\infty}(\Lambda;M,F_0,F_1)$$
consist of those integral operators on $M$ having an integral kernel belonging to $\scrC^\infty({F_1\boxtimes F_0})$,  
depending continuously on $\theta$ and rapidly decreasing on $\tau$ in the case of parameter-dependence. 

The locally defined homogeneous principal symbols induce globally bundle homomorphisms 
$\pi^*F_0\to\pi^*F_1$, where $\pi$ is the natural projection $T^*M\setminus\{0\}\to M$ in case of 
operators without parameter, otherwise the projection $(T^*M\times\Lambda)\setminus\{0\}\to M$. 

\section{Abstract parameter-dependent pseudodifferential operators}\label{sec:2}

In this section we recall and summarize the concept of abstract pseudodifferential operators of Toeplitz type 
from \cite{Seil}. We shall use a slightly modified and simplified notation. 

\subsection{Abstract pseudodifferential calculi}

Let $G$ denote a ``set of admissable weights $g$" and set $\frakG=G\times G$. 
With any $g\in G$ there is associated a Hilbert space $H(g)$. With any pair $\frakg=(g_0,g_1)\in \frakG$ we associate 
a vector space $L^0(\Lambda;\frakg)$ of ``zero-order parameter-dependent operators" and a subspace 
$L^{-\infty}(\Lambda;\frakg)$ of ``smoothing parameter-dependent operators", where 
 $$L^0(\Lambda;\frakg)\subset\scrC\big(\Lambda,\scrL(H(g_0),H(g_1))\big),$$
and
the smoothing operators, additionally, are assumed to vanish as $|\lambda|\to+\infty$. To emphasize the presence
of the parameter we shall use notations $A(\lambda)$, $B(\lambda)$, etc. for the elements of $L^0(\Lambda;\frakg)$. 

If $\frakg_0=(g_0,g_1)$ and $\frakg_1=(g_1,g_2)$ are two pairs of admissible weights and 
$\frakg_1\circ\frakg_0:=(g_0,g_2)$, then the $(\lambda-$wise$)$ composition of operators is assumed to induce maps 
    $$L^{\mu_1}(\Lambda;\frakg_1)\times L^{\mu_0}(\Lambda;\frakg_0)\lra 
        L^{\mu_1+\mu_0}(\Lambda;\frakg_1\circ\frakg_0),$$
for any choice of $\mu_0,\mu_1\in\{0,-\infty\}$. If $\frakg=(g_0,g_1)$ and $\frakg^{(-1)}:=(g_0,g_1)$ then taking 
$(\lambda-$wise$)$ the adjoint of operators is supposed to yield mapppings 
 $$L^\mu(\Lambda;\frakg)\lra L^\mu(\Lambda;\frakg^{(-1)}).$$
Due to the vanishing at infinity of the smoothing operators, $1-R(\lambda)$ is invertible for sufficiently large 
$|\lambda|$ whenever $R(\lambda)\in L^{-\infty}(\Lambda;\frakg)$ with $\frakg=(g,g)$. We shall assume that the 
inverse again has the same structure, i.e., there exists an $S(\lambda)\in L^{-\infty}(\Lambda;\frakg)$ such that 
 $$(1-R(\lambda))(1-S(\lambda))=(1-S(\lambda))(1-R(\lambda))=1$$
for $|\lambda|$ sufficiently large. This is equivalent to asking that there exists an 
$R^\prime(\lambda)\in L^{-\infty}(\Lambda;\frakg)$ such that, for large $|\lambda|$,  
 $$R^\prime(\lambda)=R(\lambda)(1-R(\lambda))^{-1}R(\lambda).$$

We shall assume that there exists a ``principal symbol", i.e., a map 
 $$A(\lambda)\mapsto \sigma(A)=\big(\sigma_1(A),\ldots,\sigma_n(A)\big)$$
assigning to each $A(\lambda)\in L^{0}(\Lambda;\frakg)$, $\frakg=(g_0,g_1)\in\frakG$, an $n$-tuple of 
$($continuous$)$ bundle homomorphisms 
\begin{equation}\label{eq:bundle}
 \sigma_k(A):E_k(g_0)\lra E_k(g_1),\qquad k=1,\ldots,n,
\end{equation}
where $E_k(g)$ denotes a Hilbert space bundle\footnote{over some base manifold $B_k(g)$, locally modelled with a 
finite or infinite dimensional separable Hilbert space} associated with the weight $g\in G$. The principal symbol is 
supposed to vanish on smoothing operators and to be compatible with addition, composition, and taking the adjoint. 
Moreover, the following are supposed to be equivalent: 
\begin{itemize}
 \item[(L1)] $A(\lambda)\in  L^{0}(\Lambda;\frakg)$ is \emph{parameter-elliptic}, i.e., all maps \eqref{eq:bundle} 
   are isomorphisms. 
 \item[(L2)] $A(\lambda)\in  L^{0}(\Lambda;\frakg)$ has a parametrix 
  $B(\lambda)\in  L^{0}(\Lambda;\frakg^{(-1)})$, i.e., 
  both $1-A(\lambda)B(\lambda)$ and $1-B(\lambda)A(\lambda)$ are smoothing. 
\end{itemize}
Due to the above described assumption on smoothing operators, for a parameter-elliptic $A(\lambda)$ we can find 
a parametrix which coincides with $A(\lambda)^{-1}$ for sufficiently large $|\lambda|$. 

\begin{example}\label{ex:0}
Let $M$ be a smooth closed Riemannian manifold. Let the set $G$ of admissable weights consist of all pairs 
$g=(M,F)$, where $F$ is a smooth hermitian vector bundle over $M$. For $g=(M,F)$ we define  
 $$H(g):= L^2(M,F), \qquad E(g)=E_1(g):=\pi^*F,$$
where $\pi:(T^*M\times\Lambda)\setminus\{0\}\to M$ is the canonical projection. 
If $\frakg=(g_0,g_1)$ with $g_j=(M,F_j)$ we let $L^0(\Lambda;\frakg)$ denote the space of zero order 
parameter-dependent classical pseudodifferential operators acting from sections of $F_0$ to sections of 
$F_1$ as described in Section $\mathrm{\ref{sec:1.2}}$. The principal symbol $\sigma(A)=\sigma_1(A)$ of 
$A(\lambda)\in L^0(\Lambda;\frakg)$ is the homogeneous principal symbol. 
\end{example}

\subsection{Parameter-dependent operators of Toeplitz type}

Let $\frakg=(g_0,g_1)$ be a pair of admissable weights and $P_j(\lambda)\in L^0(\Lambda;\frakg_j)$, 
$\frakg_j=(g_j,g_j)$, with $j=0,1$ be two projections, i.e., $P_j(\lambda)^2=P_j(\lambda)$. We then set 
 $$T^\mu(\Lambda;\frakg,P_0,P_1)=
     \Big\{A(\lambda)\in L^\mu(\Lambda;\frakg) \;\big\vert\; A(\lambda)(1-P_0(\lambda))=0,
     \;(1-P_1(\lambda))A(\lambda)=0\Big\}$$
with $\mu=0$ or $\mu=-\infty$. Note that this implies $A(\lambda)=P_1(\lambda)A(\lambda)P_0(\lambda)$ whenever 
$A(\lambda)\in T^\mu(\Lambda;\frakg,P_0,P_1)$. A parametrix of such an $A(\lambda)$ is any parameter-dependent 
operator   $B(\lambda)\in L^{0}(\Lambda;\frakg^{(-1)},P_1,P_0)$ such that 
\begin{align*}
 P_0(\lambda)-B(\lambda)A(\lambda)&\in T^{-\infty}(\Lambda;\frakg_0,P_0,P_0)\\  
 P_1(\lambda)-A(\lambda)B(\lambda)&\in T^{-\infty}(\Lambda;\frakg_1,P_1,P_1).
\end{align*}

If $P(\lambda)\in L^0(\Lambda;\frakg)$, $\frakg=(g,g)$, is a projection then so is any associated principal symbol 
$\sigma_k(P)$. Hence 
 $$E_k(g,P):=\mathrm{range}\Big(\sigma_k(P):E_k(g)\to E_k(g)\Big)$$
is a subbundle of $E_k(g)$. For $A(\lambda)\in T^0(\Lambda;\frakg,P_0,P_1)$ we then define the principal symbol 
 $$\sigma(A;P_0,P_1)=\big(\sigma_1(A;P_0,P_1),\ldots,\sigma_n(A;P_0,P_1)\big)$$
with 
\begin{equation}\label{eq:bundle2}
 \sigma_k(A;P_0,P_1)=\sigma_k(A):E_k(g_0,P_0)\lra E_k(g_1,P_1),\qquad k=1,\ldots,n. 
\end{equation}
The following theorem now holds true: 

\begin{theorem}[Theorem 3.18 of \cite{Seil}]\label{thm:main1}
Under the above assumptions the following two properties are equivalent: 
\begin{itemize}
 \item[$\mathrm{(T1)}$] 
  $A(\lambda)\in T^{0}(\Lambda;\frakg,P_0,P_1)$ is parameter-elliptic, i.e., all maps \eqref{eq:bundle2} are 
   isomorphisms. 
 \item[$\mathrm{(T2)}$] 
  $A(\lambda)\in  T^{0}(\Lambda;\frakg,P_0,P_1)$ has a parametrix 
  $B(\lambda)\in T^{0}(\Lambda;\frakg^{(-1)},P_1,P_0)$.
\end{itemize}
In this case, one can choose a parametrix $B(\lambda)$ in such a way that 
$B(\lambda)A(\lambda)=P_0(\lambda)$ and $A(\lambda)B(\lambda)=P_1(\lambda)$ for large enough $|\lambda|$.
\end{theorem} 

Let us note that if we do not have equivalence of (L1) and (L2) but only that (L1) implies (L2) then (T1) implies (T2). 
Clearly Theorem \ref{thm:main1} implies that 
 $$A(\lambda):H_\lambda(g_0,P_0)\lra H_\lambda(g_1,P_1)$$
is an isomorphism for large $|\lambda|$, where we have used the notation 
 $$H_\lambda(g,P)=P(\lambda)\big(H(g)\big),\qquad P(\lambda)\in L^0(\Lambda;(g,g)).$$
Note that $H_\lambda(g,P)$ is a closed subspace of $H(g)$ for any $\lambda$.  

\begin{example}
We can apply the above construction to the parameter-dependent classical pseudodifferential operators on $M$ 
as described in Example \ref{ex:0}. However, in this way we cannot deal in a satisfactory manner with Toeplitz 
operators of the form $P_1A(\lambda)P_0$ where at least one of the projections does not depend on the 
parameter.  This is due to the fact that $L^0_\cl(M,F,F)\not\subset L^0_\cl(\Lambda;M,F,F)$; in fact 
 $$L^0_\cl(M,F_0,F_1)\cap L^0_\cl(\Lambda;M,F_0,F_1)=\mathrm{Hom}(F_0,F_1),$$
the space of bundle homomorphisms $F_0\to F_1$. In the next section we develop a calculus avoiding this problem. 
\end{example}

Let us remark that Grubb in \cite{Grub} introduced a pseudodifferential calculus that allows to consider 
fixed operators 
$($i.e., operators not depending on the parameter$)$ as parameter-dependent ones. In this calculus to each 
element is asigned an order as well as a regularity (we do not go into details here, but only mention that the 
regularity somehow measures the deviation of the parameter-dependence from the one described in 
Section \ref{sec:1}). There is a concept of parameter-ellipticity in this calculus requiring, in particular, positive 
regularity. Since fixed operators in general only have regularity 0 this calculus is not $($or, better, only to some 
extend$)$ suited for our purpose. 

\section{A calculus of parameter-dependent operators on $\rz^n$}\label{sec:3}

In $\rz^n$ we have a ono-to-one correspondence between pseudodifferential operators $A=\op(a)$ and 
their symbols. For this reason we shall mainly work on the level of symbols. The composition of operators 
corresponds to the Leibniz product of symbols, defined by the (oscillatory) integral
 $$(a\#b)(x,\xi)=\iint e^{-iy\eta}a(x,\xi+\eta)b(x+y,\xi)\,dy\dbar\eta,$$
while taking the $($formal$)$ adjoint of $A$ with respect to the $L_2$-scalar product corresponds to passing to the 
symbol 
 $$a^{(*)}(x,\xi)=\iint e^{-iy\eta}a(x+y,\xi)^*\,dy\dbar\eta.$$
Note that both maps 
$(a,b)\mapsto a\#b:S^{\mu_1}(N_1,N_2)\times S^{\mu_0}(N_0,N_1)\to S^{\mu_0+\mu_1}(N_0,N_2)$ 
and 
$a\mapsto a^{(*)}:S^\mu(N_0,N_1)\to S^\mu(N_1,N_0)$ are continuous. 

\subsection{A first calculus of parameter-dependent symbols}\label{sec:3.1}

\begin{definition}\label{def:calc}
With real $\mu$ or $\mu=-\infty$ let $\bfS^\mu(\Lambda)=\bfS^\mu(\Lambda;N_0,N_1)$ denote 
the space of all continuous and bounded functions $a:\Lambda\to S^\mu$ for whom there exists a continuous function 
$a^\infty:[a,b]\to S^\mu$ such that 
\begin{equation}\label{eq:limit}
 a(\tau,\cdot)\xrightarrow{\tau\to\infty}a^\infty\quad\text{ in }\scrC([a,b],S^{\mu+1}) 
\end{equation}
$($note the order $\mu+1$ in \eqref{eq:limit}$)$. We call $a^\infty$ the \emph{limit-family} of $a$. 
By $\bfS^\mu_0(\Lambda)$ we denote the subspace of symbols with vanishing limit-family.
\end{definition}

Obviously $S^\mu\subset \bfS^\mu(\Lambda)$. The calculus is closed under composition (Leibniz product)  and 
taking the  adjoint. The limit-family behaves multiplicative under composition. Composition with the symbol 
$\spk{\xi}^\nu$ (both from the left or the right) yields isomorphisms $\bfS^\mu(\Lambda)\to \bfS^{\mu+\nu}(\Lambda)$. 

Let us remark that in the previous definition it would be equivalent to ask only for 
$a^\infty\in\scrC([a,b],S^{\mu+1})$, since it then follows that $a^\infty\in\scrC([a,b],S^{\mu})$. 
In fact, this is true, since we can identify $\scrC([a,b],S^{\nu})$ with $S^\nu(E)$ for $E=\scrC([a,b])$ and 
then use following general observation:  

\begin{lemma}\label{lem:conv}
Let $E$ be a Fr\'{e}chet space and let $(a_n)$ be a bounded sequence in $S^\mu(E)$ that converges in 
$S^{\mu+1}(E)$ to $a^\infty$. Then $a^\infty\in S^{\mu}(E)$.  
\end{lemma}
\begin{proof}
Since $S^\mu(E)\hookrightarrow\scrC^\infty(\rz^n\times\rz^n,E)$ continuously, $a^\infty$ is a smooth function 
with values in $E$. Moreover, for any semi-norm $p$ of $E$ we have 
 $$p\big(D^\alpha_\xi D^\beta_x a^\infty(x,\xi)\big)\xleftarrow{n\to\infty}
     p\big(D^\alpha_\xi D^\beta_x a_n(x,\xi)\big)\le C\spk{\xi}^{\mu-|\alpha|}$$
with a constant $C$ independent of $(x,\xi)$ and $n$. This yields the claim. 
\end{proof}

The next lemma states that asymptotic summation is possible within the calculus. 

\begin{lemma}\label{lem:asymptotic}
Given $a_{\mu-k}\in \bfS^{\mu-k}(\Lambda)$, 
$k\in\nz_0$, there exists an $a\in \bfS^\mu(\Lambda)$ such that 
$a-\sum\limits_{k=0}^{\ell-1}a_{\mu-k}$ belongs to $\bfS^{\mu-\ell}(\Lambda)$ for any $\ell$.
\end{lemma}
\begin{proof}
Let $K$ be $\rz_+\cup\{\infty\}$ the one-point completion of $\rz_+$ at infinity and  
\begin{align*}
 E_0=\scrC_b(\Lambda),\qquad E_1=\scrC(K\times[a,b]), 
\end{align*}
where $\scrC_b(\Lambda)$ is the space of continuous and bounded functions on $\Lambda$. 
With a fixed zero excision function $\chi(\xi)$ we can choose a zero sequence $(c_k)$ such that any series 
 $$b_\ell(x,\xi,\lambda):=\sum_{k=\ell}^\infty \chi(c_k\xi)a_{\mu-k}(x,\xi,\lambda),\qquad \ell\in\nz_0,$$
converges both in $S^{\mu-\ell}(E_0)$ and $S^{\mu+1-\ell}(E_1)$ for any $\ell$. Since taking the limit 
$($i.e., evaluation in $\tau=\infty)$ is a continuous map $E_1\to E_2:=\scrC([a,b])$, we see that 
 $$\lim_{\tau\to\infty}b_\ell(x,\xi,\tau,\theta)=\sum_{k=\ell}^\infty \chi(c_k\xi)a_{\mu-k}^\infty(x,\xi,\theta)$$
with convergence in $S^{\mu-\ell+1}(E_2)$. However we can modify (i.e., diminish) the $c_k$ in such a way that 
the above convergences remain valid and, additionally, the last series converges in $C([a,b],S^{\mu-\ell})$ for any 
$\ell$. Hence we can take $a=b_0$. 
\end{proof}

\begin{example}\label{ex:1}
Let $a:\Lambda\to S^\mu$ be continuous and bounded. Assume that $a$ considered as a function with values 
in $S^{\mu+1}$ is continuously differentiable with respect to $\tau$ and that there exists a $\delta>0$ such that 
$(1+\tau)^{1+\delta}\partial_\tau a$ is bounded. Then $a\in \bfS^\mu(\Lambda)$. In fact, \eqref{eq:limit} holds 
true if we set 
 $$\displaystyle a^\infty=a(1)+\int_1^\infty \partial_\tau a(\tau,\cdot)\,d\tau$$
$($observe Lemma $\mathrm{\ref{lem:conv}}$ and the comment given before$)$. 
\end{example}

\begin{example}\label{ex:2}
Let $a\in S^0_\cl(\Lambda)$ be a classical symbol of order $\mu=0$. Then $a$ belongs to $\bfS^0(\Lambda)$ and 
has limit-family 
 $$a^\infty(x,\theta)=a^{(0)}(x,0,1,\theta)\qquad\text{$($independent of $\xi)$},$$
where $a^{(0)}$ is the homogeneous principal symbol of $a$. In fact, first it is easy to see that 
analogous symbols of order $-1$ belong to $\bfS^0(\Lambda)$ and have vanishing limit-family. Thus we can 
assume that $a$ is homogeneous in the large, i.e., satisfies 
  $$a(x,t\xi,t\tau,\theta)=a(x,\xi,\tau,\theta)\qquad \forall\;t\ge 1\quad\forall\; |(\xi,\tau)|\ge 1.$$
Writing $a(x,\xi,\tau,\theta)=a(x,\xi/\tau,1,\theta)$ for $\tau\ge1$, we 
see that the limit-family must be as stated if it exists. Moreover, 
 $$\tau^2\partial_\tau a(x,\xi,\tau,\theta)
    =-\smsum_{j=1}^n\xi_j(\partial_{\xi_j}a)(x,\xi/\tau,1,\theta) 
    =-\smsum_{j=1}^n\xi_j\tau(\partial_{\xi_j}a)(x,\xi,\tau,\theta).$$
Using that $\partial_{\xi_j}a\in S^{-1}_\cl(\Lambda)$ it follows easily that $\tau^2\partial_\tau a$ is bounded with 
values in $S^1$ and we can apply the previous example. 
\end{example}

\subsubsection{Parameter-ellipticity and invertibility for large parameters}\label{sec:3.2}

Let us now discuss the concept of ellipticity for the classes $\bfS^\mu(\Lambda)$. Throughout this subsection we 
assume $N_0=N_1=N$ for some $N\in\nz$. Let us first observe the following: 

\begin{proposition}\label{prop:1-R}
Let $r\in \bfS^{-\infty}_0(\Lambda)$. Then there exists an $s\in \bfS^{-\infty}_0(\Lambda)$ such that 
 $$(1+r(\lambda))\#(1+s(\lambda))=(1+s(\lambda))\#(1+r(\lambda))=1$$
provided $|\lambda|$ is large enough. In particular, $1+r(\lambda)$ is invertible in $S^0$ for large $|\lambda|$. 
\end{proposition}
\begin{proof}
Let us recall that the set of invertible elements in $S^0$ form an open set and that inversion is a continuous 
map. Since $r(\lambda)\to 0$ in $S^0$ for 
$|\lambda|\to\infty$, we can conclude the existence of a $C$ such that $(1-r(\lambda))^{-1}$ exists in $S^0$ 
for $|\lambda|\ge C$ and is continuous and bounded as a function of $\lambda$. Whenever the inverse exists,  
 $$(1+r(\lambda))^{-1}=1-r(\lambda)+r(\lambda)\#(1+r(\lambda))^{-1}\# r(\lambda).$$
Thus if $\chi(t)$ is a zero-excision function vanishing for $t\le C$ then 
 $$s(\lambda)=-r(\lambda)+\chi(\tau)r(\lambda)\#(1+r(\lambda))^{-1}\#r(\lambda)$$
is the desired element of $\bfS^{-\infty}_0(\Lambda)$. 
\end{proof}

We shall call $a\in \bfS^\mu(\Lambda)$ \emph{parameter-elliptic} provided the following two conditions hold: 
\begin{itemize}
 \item[$($I$)$] There exists a $C\ge0$ such that $a(x,\xi,\lambda)$ is invertible whenever $|\xi|\ge C$ and 
  $|a(x,\xi,\lambda)^{-1}|\spk{\xi}^\mu$ is uniformly bounded in $x\in\rz^n$,   $\lambda\in\Lambda$, and 
  $|\xi|\ge C$. 
 \item[$($II$)$] The limit family $a^\infty$ is invertible in $S^\mu$ $($pointwise for each $\theta\in[a,b])$. 
\end{itemize}

Note that in $($II$)$ we could equivalently ask that $a^\infty$ is pointwise invertible as a map 
$H^s(\rz^n)\to H^{s-\mu}(\rz^n)$ for some $s\in\rz$.    

\begin{proposition}
Let $a\in \bfS^\mu(\Lambda)$ be parameter-elliptic as described above. Then we can choose a zero excision 
function $\chi(\xi)$ vanishing for $|\xi|\le C$ such that $b(x,\xi,\lambda):=\chi(\xi)a(x,\xi,\lambda)^{-1}$
belongs to $\bfS^{-\mu}(\Lambda)$ and both $a(\lambda)\#b(\lambda)-1$ and $b(\lambda)\#a(\lambda)-1$ belong 
to $\bfS^{-1}(\Lambda)$. 
\end{proposition}
\begin{proof}
Without loss of generality $\mu=0$. Let $a^\infty\in S^0$ be the limit-family. The continuity of inversion in $S^0$ 
yields the existence of a $D\ge0$ such that $a^\infty(x,\xi,\theta)$ is invertible for $|\xi|\ge D$ with inverse uniformly 
bounded in $(x,\xi,\theta)$. By enlarging one constant or the other we may assume $C=D$. With  
$b^\infty(x,\xi,\theta)=\chi(\xi)a^\infty(x,\xi,\theta)^{-1}$ we have 
\begin{align*}
\begin{split}
 b(x,\xi,\tau,\theta)&-b^\infty(x,\xi,\theta)\\
 &=\chi(\xi)a(x,\xi,\tau,\theta)^{-1}\big(a^\infty(x,\xi,\theta)-a(x,\xi,\tau,\theta)\big)a^\infty(x,\xi,\theta)^{-1}.
\end{split}
\end{align*}
This implies that $b(\tau,\cdot)\xrightarrow{\tau\to\infty}b^\infty$ in $\scrC([a,b],S^1)$. Thus $b\in \bfS^0(\Lambda)$. 
The remaining claim follows by using the explicit formula for the remainder term $r:=a\#b-ab$, namely 
 $$r(x,\xi,\lambda)=\sum_{|\alpha|=1}
     \iint e^{-iy\eta}D^\alpha_\xi a(x,\xi+\eta,\lambda)\partial^\alpha_x b(x+y,\xi,\lambda)\,dy\dbar\eta,$$
and analogously for $b\#a-ba$. 
\end{proof}

\begin{theorem}\label{thm:parametrix1}
Let $a\in \bfS^\mu(\Lambda)$ be parameter-elliptic. Then there exists a $b\in \bfS^{-\mu}(\Lambda)$ such that 
both $a(\lambda)\#b(\lambda)-1$ and $b(\lambda)\#a(\lambda)-1$ belong to $\bfS^{-\infty}_0(\Lambda)$ and 
vanish for large enough $|\lambda|$. In particular, $a(\lambda)$ is invertible in $S^0$ for sufficiently large 
$|\lambda|$ with $a(\lambda)^{-1}=b(\lambda)$. 
\end{theorem}
\begin{proof}
Using the previous proposition and the usual Neumann series argument together with Lemma \ref{lem:asymptotic} we 
can construct a parametrix $b^\prime$ modulo $\bfS^{-\infty}(\Lambda)$. 
With $r(\lambda)=1-a(\lambda)\#b^\prime(\lambda)$ we then 
have $r^\infty=1-a^\infty\#b^{\prime\infty}$. Hence, using (II),  
 $$b^{\prime\prime}(\tau,\theta):=b^\prime(\tau,\theta)+a^\infty(\theta)^{-1}r^\infty(\theta)$$
is a parametrix modulo $\bfS^{-\infty}_0(\Lambda)$. In fact, 
$b^{\prime\prime}-b^{\prime}\in \bfS^{-\infty}(\Lambda)$, since $r^\infty\in \scrC([a,b],S^{-\infty})$, and 
 $$(1-a\#b^{\prime\prime})^\infty
     =(1-a\#b^{\prime})^\infty-a^\infty(a^\infty)^{-1}r^\infty=r^\infty-r^\infty=0.$$
It remains to apply Proposition \ref{prop:1-R} to finally obtain the desired parametrix. 
\end{proof}

\begin{example}\label{ex:3}
Let $a\in S^0_\cl(\Lambda)$ be a classical symbol of order 0. Then parameter-ellipticity in the sense of 
$\mathrm{(I)}$ and $\mathrm{(II)}$ is equivalent to the usual parameter-ellipticity of classical symbols, i.e., the
principal symbol $a^{(0)}$ is everywhere invertible and there exists a $C$ such that 
 $$|a^{(0)}(x,\xi,\lambda)^{-1}|\le C\qquad x\in\rz^n,\quad |(\xi,\lambda)|=1.$$
Note that the latter estimate is only a condition for $|x|\to\infty$ and can be omitted for example when $a^{(0)}$ 
is constant in $x$ for large $|x|$. 
\end{example}

\subsection{Weakly classical symbols  and parameter-ellipticity}\label{sec:3.3}

The calculus described above is complete but still not suited for our purposes, since the ellipticity condition $($I$)$ 
not only asks the invertibility of a certain principal symbol but also requires an estimate on the inverted symbol. 
Our next aim is to single out a subcalculus in which $($I$)$ can be replaced by a condition avoiding such kind of 
estimates. This calculus contains $S^\mu_\cl$ as well as $S^\mu_\cl(\Lambda)$, the latter only in case of $-\mu$ 
being a non-negative integer.  

\subsubsection{A class of homogeneous symbols}\label{sec:3.3.1}

Let 
 $$\sz^n_+=\big\{(\bar\xi,\bar\tau)\,\vert\,\bar\xi\in\rz^n,\;\bar\tau\ge0,\;|(\bar\xi,\bar\tau)|=1\big\}$$
be the closed upper semi-sphere in ${\rz}^{n+1}$ and, with $\bfn:=(0,1)$ the ``north-pole", 
\begin{equation*}
 \whsz^n_+=\sz^n_+\setminus \{\bfn\}.
\end{equation*}
On $\whsz^n_+$ we shall make use of polar-coordinates 
 $$\bar{\xi}=\sin\rho\cdot\phi,\qquad \bar{\tau}=\cos\rho\qquad (\phi\in\sz^{n-1},\;0<\rho\le\pi/2),$$
respectively $\phi=\bar{\xi}/|\bar{\xi}|$ and $\rho=\arccos\bar{\tau}$. 
If $E$ is a Fr\'{e}chet space we let 
 $$\scrC^{\infty,\gamma}(\whsz^n_+,E), \qquad\gamma\in\rz,$$ 
denote the space of all smooth $E$-valued functions $\wh{a}$ on $\whsz^n_+$ having the property that 
$\rho^{-\gamma+\eps}(\rho\partial_\rho)^j\Delta_\phi^k \wh{a}$ is bounded on $\whsz^n_+$ for any choice 
of $j,k\in\nz_0$ and $\eps>0$. 

\begin{definition}\label{def:taylor1}
We say that $\wh{a}\in\scrC^{\infty}(\whsz^n_+,E)$ has \emph{$($generalized$)$ Taylor asymptotics} at the 
north-pole $\bfn$ if 
\begin{equation}\label{eq:taylor1}
 \wh{a}(\phi,\rho)\sim\sum_{j=0}^\infty \rho^j\,\wh{a}_j(\phi),\qquad \wh{a}_j\in\scrC^\infty(\sz^{n-1},E), 
\end{equation}
for suitable $\wh{a}_j$, where $\sim$ means that, for any $\ell\in\nz$,  
\begin{equation}\label{eq:taylor2}
 \wh{a}(\phi,\rho)-\sum_{j=0}^{\ell} \rho^j\,\wh{a}_j(\phi)\;\in\;\scrC^{\infty,\ell+1}(\whsz^n_+,E). 
\end{equation}
The space consisting of all such functions we denote by $\scrC^{\infty}_T(\whsz^n_+,E)$.
\end{definition}

\begin{remark}\label{rem:radial}
If $\wh{a}\in\scrC^{\infty}_T(\whsz^n_+,E)$ is as in \eqref{eq:taylor1} then 
 $$\wh{a}_0(\phi)=\lim_{\rho\to0}\wh{a}(\phi,\rho).$$
In particular, if $\wh{a}\in\scrC^{\infty}(\sz^n_+,E)$ is smooth on the whole upper semi-sphere then 
 $$\wh{a}_0(\phi)=\wh{a}(0,1)=\wh{a}(\bfn)$$
is constant in $\phi$ and coincides with the value of $\wh{a}$ in the north-pole $\bfn$.  
\end{remark}

We shall now choose, both in the above definitions as well as in the remaining part of this section, 
the space $E$ to be $\scrC\big([a,b]_\theta,\scrC^\infty_b(\rz^n_x)\big)$. 
The resulting $E$-valued symbols we shall consider as functions of the variables $(x,\xi,\tau,\theta)$. 
 
\begin{definition}\label{def:wcl}
With $\bfS^{(\mu)}(\Lambda)$, $\mu\in\rz$, we shall denote the space of all functions on 
$\rz^n_x\times(\rz^n_\xi\setminus\{0\})\times\Lambda$ of the form 
\begin{equation}\label{eq:taylor3}
 a^{(\mu)}(x,\xi,\tau,\theta)=|\xi|^\mu\,\wh{a}\Big(x,\frac{(\xi,\tau)}{|(\xi,\tau)|},\theta\Big),
 \qquad \wh{a}\in\scrC^{\infty}_T(\whsz^n_+,E).  
\end{equation}
\end{definition}

Let us mention two particular examples: First let $b(x,\xi)$ be positively homogeneous of degree $\mu$ in $\xi\not=0$. 
If we define $\wh{a}$ by $\wh{a}(x,\bar\xi,\bar\tau,\theta)=b(x,\bar\xi/|\bar\xi|)$ then 
$a^{(\mu)}(x,\xi,\tau,\theta)=b(x,\xi)$ with $a^{(\mu)}$ from \eqref{eq:taylor3}. Moreover, 
$\wh{a}\in\scrC^{\infty}_T(\whsz^n_+,E)$, since $\wh{a}(x,\phi,\rho,\theta)=b(x,\phi)$. 
In other words, homogeneous symbols not depending on the parameter $\lambda$ can be considered as 
parameter-dependent homogeneous symbols. If $b(x,\xi,\tau,\theta)$ is positively homogeneous of degree $\mu$ in 
$(\xi,\tau)\not=0$, then $b$ can be written as in \eqref{eq:taylor3} with 
$\wh{a}(x,\bar\xi,\bar\tau,\theta)=|\bar\xi|^{-\mu}b(x,\bar\xi,\bar\tau,\theta)$.  
Since the restriction of $b$ to the unit-sphere is a smooth function (including the north-pole) and 
$|\bar\xi|^{-\mu}=(\sin\rho)^{-\mu}$, we conclude that $b$ can be viewed as an element of 
$\bfS^{(\mu)}(\Lambda)$ provided $\mu$ is an integer less or equal to zero.  

It is evident from the definition that multiplication with $|\xi|^{\nu}$  induces isomorphisms 
\begin{equation}\label{eq:taylor3b}
 a^{(\mu)}(x,\xi,\tau,\theta)\mapsto |\xi|^{\nu}a^{(\mu)}(x,\xi,\tau,\theta):\;
 \bfS^{(\mu)}(\Lambda)\lra \bfS^{(\mu+\nu)}(\Lambda),\qquad \nu\in\rz.
\end{equation}

\begin{remark}\label{rem:restriction}
Restriction to $\whsz^n_+$ of functions defined on $(\rz^n_\xi\setminus\{0\})\times\rpbar$ yields a canonical identification 
of $\bfS^{(\mu)}(\Lambda)$ with the weighted space $\rho^{\mu}\scrC^{\infty}_T(\whsz^n_+,E)$. 
In fact, the inverse of this map is given by 
 $$\alpha(x,\bar\xi,\bar\tau,\theta)\mapsto a^{(\mu)}(x,\xi,\tau,\theta):=
     |(\xi,\tau)|^\mu\,\alpha\Big(x,\frac{(\xi,\tau)}{|(\xi,\tau)|},\theta\Big);$$
choosing $\wh{a}(x,\bar\xi,\bar\tau,\theta)=|\bar\xi|^{-\mu}\alpha(x,\bar\xi,\bar\tau,\theta)$ gives the 
representation of \eqref{eq:taylor3}.
\end{remark}

\begin{theorem}\label{thm:estimate}
Let $\chi$ be a zero-excision function, $a^{(\mu)}\in \bfS^{(\mu)}(\Lambda)$ and 
 $$a(x,\xi,\tau,\theta)=\chi(\xi)\,a^{(\mu)}(x,\xi,\tau,\theta).$$
\begin{itemize}
 \item[a$)$] $a$ satisfies the uniform estimates  
   \begin{equation}\label{eq:taylor4}
     |D^\alpha_\xi D^\beta_x D^k_\tau a(x,\xi,\tau,\theta)|\le C_{\alpha\beta k}\spk{\xi}^{\mu-|\alpha|}\spk{\xi,\tau}^{-k},
     \qquad \alpha,\beta\in\nz_0^n,\quad k\in\nz_0.
   \end{equation}
 \item[b$)$] $a$ belongs to $\bfS^\mu(\Lambda)$; if $\wh{a}$ from \eqref{eq:taylor3} has an expansion as in 
   \eqref{eq:taylor1} then the limit-family of $a$ is $a^\infty(x,\xi,\theta)=\chi(\xi)|\xi|^\mu\wh{a}_0(x,\xi/|\xi|,\theta)$. 
\end{itemize}
\end{theorem}
\begin{proof}
In view of \eqref{eq:taylor3b} we may assume that $\mu=0$. For convenience of notation we assume 
independence of the $x$ and $\theta$ variables. Let us prove part a$)$: 

$($i$)$ Consider the case $\wh{a}(\phi,\rho)=\rho$, i.e., 
 $$a(\xi,\tau)=\chi(\xi)\arccos\frac{\tau}{|(\xi,\tau)|}.$$
Choose a zero-excision function $\chi_1(\xi,\tau)$ such that $\chi_1\chi=\chi$. Then 
\begin{align*}
 \partial_\tau a(\xi,\tau)
 &=-\Big(1-\frac{\tau}{|(\xi,\tau)|}\Big)^{-1/2}\Big(|(\xi,\tau)|^{-1}-\tau^2|(\xi,\tau)|^{-3}\Big)\chi(\xi)\\
 &=\chi(\xi)|\xi|\cdot \chi_1(\xi,\tau)|(\xi,\tau)|^{-2}.
\end{align*}
The first factor is a symbol in $S^1$, while the second is a symbol in $S^{-2}(\rpbar)$. 
This yields \eqref{eq:taylor4} in case $k\ge1$. Moreover, 
\begin{equation*}
 \partial_{\xi_i}\arccos\frac{\tau}{|(\xi,\tau)|}=\frac{\xi_i}{|\xi|}\cdot\frac{\tau}{\tau^2+|\xi|^2}
\end{equation*}
is a product of a function positively homogeneous of degree 0 in $\xi$ and one of degree $-1$ in $(\xi,\tau)$. 
We obtain \eqref{eq:taylor4} in case $k=0$. By product rule then \eqref{eq:taylor4} also holds for 
$\wh{a}(\phi,\rho)=\rho^j$ with $j\in\nz$. 

$($ii$)$ If $\wh{a}(\phi,\rho)=\wh{a}(\phi)$ then $a(\xi,\tau)=\chi(\xi)a(\xi/|\xi|)$ is a symbol in $S^0$ 
and \eqref{eq:taylor4} is valid. 

$($iii$)$ Let $\wh{r}\in\scrC^{\infty,\ell+1}(\whsz^n_+)$ and $A$ be a differential operator of order $k$ on 
$\sz^n_+$ with smooth coefficients. In polar-coordinates $A$ takes the form 
 $$\rho^{-k}\sum_{i=0}^k A_i(\rho)(\rho\partial_\rho)^i,\qquad 
     A_i(\rho)\in\scrC^\infty\big([0,\pi/2],\mathrm{Diff}^{k-i}(\sz^{n-1})\big).$$
Therefore $A\wh{r}=O(\rho^\eps)$ for any $A$ of order $k<\ell+1$. 
We conclude that $\wh{r}\in\scrC^{\ell}(\sz^n_+)$. It follows that 
$\chi(\xi)\wh{r}\big((\xi,\tau)/|(\xi,\tau)|\big)$ satisfies \eqref{eq:taylor4} provided $|\alpha|+k\le \ell$. 

$($iv$)$ To complete the proof it remains to combine $($i$)$-$($iii$)$ with the expansion \eqref{eq:taylor2} 
which is valid for arbitrary $\ell$. 

To prove part b$)$ it remains to verify the existence of the limit-family. Obviously we may 
assume that $\wh{a}_0=0$ and then show that the limit-family is zero. Since $\wh{a}_0=0$ we have 
$\wh{a}=\rho\wh{b}$ with $\wh{b}\in \scrC^{\infty}_T(\whsz^n_+)$. If we associate $b$ with $\wh{b}$ then 
$b\in \bfS^0(\Lambda)$. Therefore it suffices to assume $\wh{a}=\rho$ and show that the 
limit-family exists and equals zero. In this case observe that 
 $$a(\xi,\tau)=\chi(\xi)\wh{a}\Big(\frac{(\xi,\tau)}{|(\xi,\tau)|}\Big)
     =\chi(\xi)\arccos\frac{\tau}{|(\xi,\tau)|}$$
converges to zero uniformly on compact subsets of $\rz^n$ as $\tau$ tends to $+\infty$. Since $\arccos$ is 
bounded it follows that, for $\alpha=0$,  
\begin{equation}\label{eq:taylor6}
 \sup_{\xi\in\rz}|D^\alpha_\xi a(\xi,\tau)|\spk{\xi}^{|\alpha|-1}\xrightarrow{\tau\to+\infty}0.
\end{equation}
It remains to verify \eqref{eq:taylor6} for arbitrary $\alpha\not=0$. However, this follows easily from the case 
$\alpha=0$ and product rule. 
\end{proof}

\subsubsection{Weakly classical symbols}\label{sec:3.4}

We now introduce a subcalculus of $\bfS^\mu(\Lambda)$ in which ellipticity is characterized purely in terms of the 
invertibility of certain principal symbols. 

\begin{definition}\label{def:calc2}
We define $\bfS^\mu_\wcl(\Lambda)$ as the subspace of $\bfS^\mu(\Lambda)$ consisting of all symbols $a$ for 
which there exist homogeneous symbols $a^{(\mu-j)}\in \bfS^{(\mu-j)}(\Lambda)$ such that, for any $N\in\nz_0$,  
 $$a-\sum_{j=0}^{N-1}\chi(\xi)a^{(\mu-j)}\;\in\; \bfS^{-N}(\Lambda)$$
for some $($and then for any$)$ zero-excision function $\chi$. We call ${a}^{(\mu)}$ the \emph{homogeneous principal 
symbol} of $a$. The leading term in \eqref{eq:taylor1} for $\wh{a}$ associated with $a^{(\mu)}$ as in \eqref{eq:taylor3} 
we shall denote by $\wh{a}^{(\mu)}_0$ and refer to as the \emph{principal angular symbol} of $a$. 
Moreover, we ask that the homogeneous principal symbol does not depend on the $x$-variable for large 
$|x|$.\footnote{This is a technical assumption to cope with the non-compactness of $\rz^n$. }
\end{definition}

Obviously, the subscript $\wcl$ stands for ``weakly classical". As in the standard case, 
\begin{equation*}
 a^{(\mu)}(x,\xi,\tau,\theta)=\lim_{t\to\infty}t^{-\mu}a(x,t\xi,t\tau,\theta)
 \qquad(\xi\not=0). 
\end{equation*}
In particular, the homogeneous components of a weakly classical symbol are uniquely determined. 
Remark $\mathrm{\ref{rem:radial}}$ and $\rho/\sin\rho\xrightarrow{\rho\to 0}1$ yields that  
\begin{equation*}
 \wh{a}^{(\mu)}_0(x,\phi,\theta)=\lim_{\rho\to 0}\rho^{-\mu}a^{(\mu)}(x,\sin\rho\cdot\phi,\cos\rho,\theta)
 \qquad (\phi\in\sz^{n-1}).
\end{equation*}
In particular, the radial principal symbol is a well-defined object and is independent of the $x$-variable for large $|x|$.  

The class of weakly classical symbols is closed under composition and taking the adjoint (of the associated operators). 
Both homogeneous principal symbol and angular principal symbol behave multiplicatively under composition. 

\begin{remark}
As a consequence of the considerations after Definition $\mathrm{\ref{def:wcl}}$, 
$\bfS^\mu_\wcl(\Lambda)$ contains $S^\mu_\cl$ and, in case $-\mu\in\nz_0$, also $S^\mu_\cl(\Lambda)$. 
In fact, for $b\in S^\mu_\cl$ and $a\in S^\mu_\cl(\Lambda)$ the homogeneous principal symbols are the usual 
ones. The principal angular symbols are $\wh{b}^{(\mu)}_0(x,\phi,\theta)=b^{(\mu)}(x,\phi)$ and 
$\wh{a}^{(\mu)}_0(x,\phi,\theta)=0$ in case $\mu<0$ and 
$\wh{a}^{(0)}_0(x,\phi,\theta)=a^{(0)}(x,0,1,\theta)$ in case $\mu=0$. 
\end{remark}

The following lemma is easily verified: 

\begin{lemma}
For $a^{(\mu)}\in \bfS^{(\mu)}(\Lambda)$ the following are equivalent: 
\begin{itemize}
 \item[a$)$] $a^{(\mu)}$ is invertible with inverse belonging to $\bfS^{(-\mu)}(\Lambda)$.  
 \item[b$)$] Both $a^{(\mu)}$ and its angular symbol $\wh{a}^{(\mu)}_0$ are pointwise 
  invertible.\footnote{where $a^{(\mu)}$ is considered as a function defined either on 
  $\rz^n_x\times(\rz^n_\xi\setminus\{0\})\times\Lambda$ or on $\rz^n_x\times\whsz^n_+\times[a,b]$, 
  while $\wh{a}^{(\mu)}_0$ is a function defined on $\rz^n_x\times\sz^{n-1}\times[a,b]$.}
\end{itemize}
\end{lemma}

We shall call $a\in \bfS^0_\wcl(\Lambda)$ \emph{parameter-elliptic} provided the following two 
conditions hold: 
\begin{itemize}
 \item[$($S1$)$] Both the homogeneous principal symbol ${a}^{(\mu)}$ and the angular symbol 
  $\wh{a}^{(\mu)}_0$ are $($pointwise$)$ invertible. 
 \item[$($S2$)$] The limit family $a^\infty$ is pointwise invertible in $S^0$. 
\end{itemize}
Recall that $($S1$)$ is equivalent to the invertibility of  $a^{(\mu)}$ in the class $\bfS^{(-\mu)}(\Lambda)$. 
It is then clear that $($S1$)$ implies the ellipticity-condition $($I$)$ of Section \ref{sec:3.2} and we obtain a 
parametrix that coincides with the inverse for large values of the parameter: 

\begin{theorem}\label{thm:parametrix2}
Let $a\in \bfS^\mu_\wcl(\Lambda)$ be parameter-elliptic. Then there exists a $b\in \bfS^{-\mu}_\wcl(\Lambda)$ 
such that both $a(\lambda)\#b(\lambda)-1$ and $b(\lambda)\#a(\lambda)-1$ belong to $S^{-\infty}_0(\Lambda)$ 
and vanish for large enough $|\lambda|$. 
\end{theorem}

Note that a symbol $a\in S^\mu_\cl$ is parameter-elliptic as a weakly-classical symbol if and only if it has an
inverse $($with respect to the Leibniz product$)$, then belonging to $S^{-\mu}_\cl$. 
For $a\in S^0_\cl(\Lambda)$ parameter-ellipticity as a weakly-classical symbol is equivalent to the standard 
parameter-ellipticity in this symbol class. Instead, a symbol $a\in S^\mu_\cl(\Lambda)$ with negative integer 
order is never parameter-elliptic as a weakly-classical symbol. 

Let us conclude this section with a result we shall need in Section \ref{sec:4}, below. 

\begin{lemma}\label{lem:A}
Let $p\in S^\mu_\cl(N_0,N_1)$ and $r_j(\lambda)\in S^{\mu_j}_\cl(\Lambda;N_j,N_j)$, $j=0,1$, such that  
$-(\mu_0+\mu_1)\in\nz_0$. Then
 $$a(\lambda):=r_1(\lambda)\#p\# r_0(\lambda)\in \bfS^{\mu+\mu_0+\mu_1}_\wcl(\Lambda;N_0,N_1).$$ 
In particular, if $r_{j}^{(\mu_j)}$ denotes the homogeneous principal symbol of $r_j$, then 
\begin{align*}
 a^{(\mu+\mu_0+\mu_1)}&=r_{1}^{(\mu_1)}p^{(\mu)}r_{0}^{(\mu_0)},\\ 
 \wh{a}^{(\mu+\mu_0+\mu_1)}_0&=
  \begin{cases}
   r_{1}^{(\mu_1)}(x,0,1,\theta)\, p^{(\mu)}\, r_{0}^{(\mu_0)}(x,0,1,\theta)&\quad:\mu_0+\mu_1=0\\
   0&\quad:\mu_0+\mu_1<0
  \end{cases},\\
 {a}^{\infty}(\theta)&=
  \begin{cases}
   r_{1}^{(\mu_1)}(x,0,1,\theta)\#{p}\,\#r_{0}^{(\mu_0)}(x,0,1,\theta)&\quad:\mu_0+\mu_1=0\\
   0&\quad:\mu_0+\mu_1<0
  \end{cases},
\end{align*}
are the homogeneous principal symbol, the principal angular symbol, and the limit family, respectively. 
\end{lemma}
\begin{proof}
Let us first write $a(\lambda)=r_1(\lambda)\#\wt{p}(\lambda)\# \wt{r}_0(\lambda)$ 
with 
 $$\wt{p}(x,\xi,\lambda)=p(x,\xi)[\xi,\tau]^{\mu_0},\qquad 
     \wt{r}_0(x,\xi,\lambda)=[\xi,\tau]^{-\mu_0} \# r_0(x,\xi,\lambda),$$
where $[\xi,\tau]$ denotes a smooth positive function that coincides with $|(\xi,\tau)|$ outside the unit-ball. 
In particular, $\wt{r}_0\in S^0_\cl(\Lambda)\subset \bfS^0_\wcl(\Lambda)$. It remains to analize 
$r_1(\lambda)\#\wt{p}(\lambda)$. To this end recall that 
$S^\mu_\cl=\scrC^\infty_b(\rz^n_x)\,\wh{\otimes}_\pi\, S^\mu_\cl(\rz^n_\xi)$
$($the completed projective tensor-product of two Fr\'{e}chet spaces$)$ and therefore we can represent 
$p$ in the form 
 $$p(x,\xi)=\sum_{\ell=1}^\infty c_\ell \alpha_\ell(x)q_\ell(\xi)$$
with zero-sequences $(\alpha_\ell)_\ell\subset\scrC^\infty_b(\rz^n_x)$ and 
$(q_\ell)_\ell\subset S^\mu_\cl(\rz^n_\xi)$, and an absolutely summable numerical sequence $(c_\ell)_\ell$. 
But then 
 $$(r_1\#\wt{p})(x,\xi,\lambda)=\sum_{\ell=1}^\infty c_\ell \beta_\ell(x,\xi,\lambda)q_\ell(\xi),\quad 
    \beta_\ell(x,\xi,\lambda)=(r_1\#\alpha_\ell)(x,\xi,\lambda)[\xi,\tau]^{\mu_0}.$$
By continuity of the Leibniz product, $(\beta_\ell)_\ell$ is a zero-sequence in 
$S^{\mu_0+\mu_1}_\cl(\Lambda)$, hence also in $\bfS^{\mu_0+\mu_1}_\wcl(\Lambda)$. 
We conclude that the series defining $r_1(\lambda)\#\wt{p}(\lambda)$ converges in 
$\bfS^{\mu+\mu_0+\mu_1}_\wcl(\Lambda)$. The remaining claims are straightforward to verify. 
\end{proof}

\subsection{Weakly classical operators}\label{sec:3.5}

In the previous sections we worked exclusively on the level of symbols. We now pass to the corresponding spaces of 
pseudodifferential operators and write 
 $$\bfL^\mu_\wcl(\Lambda;\rz^n,N_0,N_1)=
     \Big\{A(\lambda)=\op(a)(\lambda) \;\big\vert\; a\in S^\mu_\wcl(\Lambda;N_0,N_1)\Big\}.$$
Similarly we obtain the space $\bfL^{-\infty}_0(\Lambda;\rz^n,N_0,N_1)$ using smoothing symbols. 

With $A(\lambda)=\op(a)(\lambda)\in \bfL^\mu_\wcl(\Lambda;\rz^n,N_0,N_1)$ we associate a principal symbol $\sigma(A)$ 
which has three components: The first component, $\sigma_1(A)$, is the homogeneous principal symbol of $a$, 
considered as a bundle 
morphism\footnote{by considering an $(N_1\times N_1)$-matrix valued function on some space $X$ as a bundle 
morphism between the trivial bundles $X\times\cz^{N_0}$  and $X\times\cz^{N_0}$} 
 $$ \big(\rz^n_x\times(\rz^n_\xi\setminus\{0\})\times\Lambda\big)\times\cz^{N_0}\lra 
      \big(\rz^n_x\times(\rz^n_\xi\setminus\{0\})\times\Lambda\big)\times\cz^{N_1}.$$
The second component,  $\sigma_2(A)$, is the principal angular symbol of $a$, considered as a bundle morphism
 $$ \big(\rz^n_x\times\sz^{n-1}_\xi\times[a,b]\big)\times\cz^{N_0}\lra 
      \big(\rz^n_x\times\sz^{n-1}_\xi\times[a,b]\big)\times\cz^{N_1}.$$
The third and last component,  $\sigma_3(A)$, is the operator-family associated with the limit-family of $a$, considered 
as a bundle morphism 
 $$[a,b]\times H^{s}(\rz^n)\lra [a,b]\times H^{s-\mu}(\rz^n)$$
$($with an arbitrary choice of $s)$. We can now apply in this setting the abstract approach described in 
Section \ref{sec:2} and obtain resulting classes of Toeplitz operators  
 $$\bfL^\mu_\wcl(\Lambda;\rz^n,(N_0,P_0),(N_1,P_1))=P_1(\lambda)\,\bfL^\mu_\wcl(\Lambda;\rz^n,N_0,N_1)\,
     P_0(\lambda)$$
with the corresponding definition of the principal symbol. 

\subsection{Weakly classical operators on smooth manifolds}\label{sec:3.6}

We finally indicate how to modify the above constructions to deal with operators on a  smooth closed Riemannian
manifold $M$ of dimension $n$. With hermitean vector bundles $F_0$ and $F_1$ over $M$ we define 
$\bfL^\mu(\Lambda;M,F_0,F_1)$ as the space of all continuous and bounded functions 
$A:\Lambda\to L^\mu(M,F_0,F_1)$ that have a limit-family 
 $$A^\infty \in \scrC([a,b],L^\mu(M,F_0,F_1)),$$ 
analogously defined as in Definition \ref{def:calc}, replacing $S^\mu$ by $L^\mu(M,F_0,F_1)$. 
Those $A$ having vanishing limit-family form the subspace $\bfL^\mu_0(\Lambda;M,F_0,F_1)$. 
Moreover, we let $\bfL^\mu_\wcl(\Lambda;M,F_0,F_1)$ denote the subspace of those operators 
having in local coordinates and local trivialisations a weakly classical symbol in the sense of Definition \ref{def:calc2}. 
The local homogeneous principal symbols of $A(\lambda)$ globally define a homogeneous principal symbol 
\begin{equation}\label{eq:sigma1}
 \sigma_1(A):\pi^*F_0\lra\pi^*F_1,\qquad \pi:(T^*M\setminus\{0\})\times\Lambda\to M,
\end{equation}
where $\pi$ denotes the canonical projection onto $M$. Furthermore, the local angular symbols of $A(\lambda)$ 
globally yield a morphism 
\begin{equation}\label{eq:sigma2}
 \sigma_2(A):\pi_1^*F_0\lra\pi_1^*F_1,\qquad \pi_1:S^*M\to M,
\end{equation}
where $\pi_1$ is the canonical projection of the co-sphere bundle onto $M$. 
The limit-family we consider as a morphism 
\begin{equation*}
 \sigma_3(A):[a,b]\times H^s(M,F_0)\lra [a,b]\times H^{s-\mu}(M,F_1)
\end{equation*}
between trivial bundles. Parameter-ellipticity of $A(\lambda)$ means bijectivity of all three morphisms $\sigma_j(A)$. 
Again we can apply the abstract approach of Section \ref{sec:2} to obtain resulting classes of Toeplitz operators  
\begin{equation*}
 \bfL^\mu_\wcl(\Lambda;M,(F_0,P_0),(F_1,P_1))=P_1(\lambda)\,\bfL^\mu_\wcl(\Lambda;M,F_0,F_1)\,
 P_0(\lambda)
\end{equation*}
where $P_j(\lambda)\in \bfL^0_\wcl(\Lambda;M,F_j,F_j)$ for $j=0,1$, including the corresponding notion of principal 
symbol and parameter-ellipticity. 

\section{Applications}\label{sec:4}

Throughout this section let
 $$P_j\in L^0_\cl(M,F_j,F_j),\qquad j=0,1,$$ 
be two pseudodifferential projections not depending on the parameter and 
 $$H^s(M,F_j,P_j)=P_j\big(H^s(M,F_j)\big),\qquad s\in\rz.$$
We shall consider two types of parameter-dependent operators. The first is 
 $$A_0(\lambda)=P_1\,A(\lambda)\,P_0,\qquad A(\lambda)\in L^\mu_\cl(\Lambda;M,F_0,F_1),$$ 
with arbitrary order $\mu\in\rz$. For the second type we assume $F_0=F_1=F$ and then let   
 $$A_1(\lambda)=P_1\,(\tau^\mu e^{i\theta}-A)\,P_0,\qquad A\in L^\mu_\cl(M,F,F),$$ 
with positive integer order $\mu\in\nz$ $($note that if $A$ is a differential operator then $A_1(\lambda)$ is a particular 
case of $A_0(\lambda)$; for pseudodifferential operators $A$ however this is not the case$)$.  
We shall derive criteria ensuring the invertibility of 
\begin{equation*}
 A_j(\lambda):H^s(M,F_0,P_0)\lra H^{s-\mu}(M,F_1,P_1),\qquad j=0,1, 
\end{equation*}
for $\lambda\in\Lambda$ having sufficiently large modulus. 
Let us denote by $p_j^{(0)}$ the homogeneous principal symbol of $P_j$ and, 
with $\pi$ and $\pi_1$ as in \eqref{eq:sigma1} and \eqref{eq:sigma2}, write 
 $$E^0(P_j):=p_j^{(0)}\big(\pi^*F_j\big)\subset\pi^*F_j,\qquad 
     E^1(P_j):=p_j^{(0)}\big(\pi_1^*F_j\big)\subset \pi^*_1F_j.$$

\begin{theorem}\label{thm:toeplitz}
Let $a^{(\mu)}$ be the homogeneous principal symbol of $A(\lambda)$ and
assume that the following mappings are isomorphisms: 
\begin{itemize}
 \item[$($a$)$] $p_1^{(0)}a^{(\mu)}p_0^{(0)}:E^0(P_0)\to E^0(P_1)$,
 \item[$($b$)$] $p_1^{(0)}a^{(\mu)}\big|_{(\xi,\lambda)=(0,1)}p_0^{(0)}:E^1(P_0)\to E^1(P_1)$, 
 \item[$($c$)$] $P_1a^{(\mu)}\big|_{(\xi,\lambda)=(0,1)}P_0:H^0(M,F_0,P_0)\to H^0(M,F_1,P_1)$
\end{itemize}
$($in $(\mathrm{b})$ we consider $a^{(\mu)}\big|_{(\xi,\lambda)=(0,1)}$ as a bundle homomorphism 
$\pi_1^*F_0\to \pi_1^*F_1$, 
in $(\mathrm{c})$ as a bundle homomorphism $F_0\to F_1$ that induces a map between the $L_2$-spaces$)$. 
Then there exist a $B(\lambda)\in \bfL^0_\wcl(\Lambda;M,F_1,F_0)$ and an 
$R(\lambda)\in L^{-\mu}_\cl(\Lambda;M,F_1,F_1)$ such that $B_0(\lambda):=P_0B(\lambda)R(\lambda)\,P_1$ 
is the inverse of $A_0(\lambda)$ for sufficiently large 
$\lambda$.\footnote{More precisely, both 
$B_0(\lambda)A_0(\lambda)-P_0\in \bfL^{-\infty}_0(\Lambda;M,(F_0,P_0),(F_0,P_0))$ and 
$A_0(\lambda)B_0(\lambda)-P_1\in \bfL^{-\infty}_0(\Lambda;M,(F_1,P_1),(F_1,P_1))$ vanish for large 
$|\lambda|$.} 
\end{theorem}
\begin{proof}
Let $R(\lambda)\in L^{-\mu}_\cl(\Lambda;M,F_1,F_1)$ be a reduction of orders, i.e., 
there exists an $S(\lambda)\in L^{\mu}_\cl(\Lambda;M,F_1,F_1)$ such that 
$S(\lambda)R(\lambda)=R(\lambda)S(\lambda)=1$ for all $\lambda\in\Lambda$. Then 
 $$\wt{A}(\lambda):=R(\lambda)A_0(\lambda)  \in L^0_\cl(\Lambda;M,F_0,F_1)\subset
     \bfL^0_\wcl(\Lambda;M,F_0,F_1)$$
and, due to Lemma \ref{lem:A}, 
 $$\wt{P}_1(\lambda):=R(\lambda)P_1S(\lambda)\;\in\;\bfL^0_\wcl(\Lambda;M,F_1,F_1).$$
Obviously, $\wt{P}_1(\lambda)$ is a projection. The assumptions $($a$)$, $($b$)$ and $($c$)$ in 
Theorem \ref{thm:toeplitz} now imply that 
\begin{equation*}
 \wt{P}_1(\lambda)\wt{A}(\lambda)P_0\;\in\; 
 \bfT^0_\wcl(\Lambda;M,(F_0,P_0),(F_1,\wt{P_1}))
\end{equation*}
is a parameter-elliptic Toeplitz operator $($the homogeneous principal symbol and the angular symbol are covered by 
$(\mathrm{a})$ and $(\mathrm{b})$, respectively, while $(\mathrm{c})$ covers the limit-family, cf. Lemma \ref{lem:A}.b$))$. 
Thus we find a parametrix 
 $$P_0{B}(\lambda)\wt{P}_1(\lambda)
     \;\in\; \bfT^0_\wcl(\Lambda;M,(F_1,\wt{P}_1),(F_0,P_0))$$ 
that coincides with the inverse for large parameter. Then 
 $$B_0(\lambda)A_0(\lambda)=\big(P_0B\wt{P}_1(\lambda)\big)\big(\wt{P}_1(\lambda)\wt{A}(\lambda)P_0\big)=P_0$$
for large $|\lambda|$, as well as 
\begin{align*}
 A_0(\lambda)B_0(\lambda)
 =S(\lambda)\big(\wt{P}_1(\lambda)\wt{A}(\lambda)P_0\big)\big(P_0B\wt{P}_1(\lambda)\big)R(\lambda)
 =S(\lambda)\wt{P}_1(\lambda)R(\lambda)=P_1.
\end{align*}
This finishes the proof. 
\end{proof}

Similarly, one can derive the following result $($which in case of $A$ being differential is a special case of Theorem 
\ref{thm:toeplitz}, above$):$

\begin{theorem}\label{thm:toeplitz2}
Let $a^{(\mu)}$ be the homogeneous principal symbol of $A$ and
assume that the following mappings are isomorphisms: 
\begin{itemize}
 \item[$($a$)$] $p_1^{(0)}(\lambda-a^{(\mu)})p_0^{(0)}:E^0(P_0)\to E^0(P_1)$,
 \item[$($b$)$] $P_1P_0:H^0(M,F,P_0)\to H^0(M,F,P_1)$. 
\end{itemize}
Then there exist a $B(\lambda)\in \bfL^0_\wcl(\Lambda;M,F,F)$ and an 
$R(\lambda)\in L^{-\mu}_\cl(\Lambda;M,F,F)$ such that $B_1(\lambda):=P_0B(\lambda)R(\lambda)\,P_1$ 
is the inverse of $A_1(\lambda)$ for sufficiently large 
$\lambda$.\footnote{Compare the previous footnote for a more detailed statement.}
\end{theorem}

In fact, the proof is analogous to that of Theorem \ref{thm:toeplitz}. Again let $R(\lambda)$ be a reduction 
of orders with inverse $S(\lambda)$. With $\wt{P}_1(\lambda)=R(\lambda)P_1S(\lambda)$ and 
$\wt{A}(\lambda)=R(\lambda)A_1(\lambda)$ one verifies that $\wt{P}_1(\lambda)\wt{A}(\lambda)P_0$ is  
parameter-elliptic in $\bfL^0_\wcl(\Lambda;M,(F,\wt{P}_1),(F,P_0))$. To this end note that 
$\wt{A}(\lambda)$ has homogeneous principal symbol $r^{(-\mu)}(\tau^\mu e^{i\theta}-a^{(\mu)})$, 
angular symbol $r^{(-\mu)}\big|_{(\xi,\tau)=(0,1)}e^{i\theta}$ and limit-family  
$r^{(-\mu)}\big|_{(\xi,\tau)=(0,1)}e^{i\theta}$. 

\subsection{The resolvent of a pseudodifferential operator in projected subspaces}\label{sec:4.2} 

Let us consider above the particular case $P_0=P_1=:P$. Then $A\in L^\mu_\cl(M,F,F)$ induces 
a densily defined, unbounded operator   
\begin{equation*}
 A_P:=PAP:\scrC^\infty(M,F,P)\subset H^s(M,F,P)\lra H^s(M,F,P),
\end{equation*}
where $s\in\rz$ is arbitrary and $\scrC^\infty(M,F,P):=P\big(\scrC^\infty(M,F)\big)$. 
Given the strip $\Lambda$ let us write 
$\Lambda^\wedge=\{z=\tau e^{i\theta}\,\vert\, (\tau,\theta)\in\Lambda\}$. 
Then the above Theorem \ref{thm:toeplitz2} has the following corollary$:$ 

\begin{corollary}
Let $a^{(\mu)}$ and $p^{(0)}$ be the homogeneous principal symbol of $A$ and $P$, respectively. If 
\begin{equation}\label{eq:spec}
 p^{(0)}a^{(\mu)}p^{(0)}:E(P)\lra E(P),\qquad E(P)=p^{(0)}(\pi_1^* F),
\end{equation}
has $($fibrewise$)$ no spectrum in $\Lambda^\wedge$, then $A_P$ has a unique closed extension, 
given by the action 
of $A_P$ on the domain $H^{s+\mu}(M,F,P)$. Denoting this extension again by $A_P$, the resolvent 
$(z-A_P)^{-1}$ exists for sufficiently large $z\in\Lambda^\wedge$ and satisfies the uniform estimate
 $$\|(z-A_P)^{-1}\|_{\scrL(H^s(M,F,P))}\le C |z|^{-1}$$
with some constant $C$. In particular, if $\Lambda^\wedge$ contains the right complex half-plane, then $A_P$ 
is the generator of an analytic semi-group. 
\end{corollary}
\begin{proof}
It is clear that $H^{s+\mu}(M,F,P)$ belongs to the domain of the closure of $A_P$. Now let 
$u$ belong to the maximal domain, i.e., both $u$ and $A_Pu=PAPu$ belong to $H^{s}(M,F,P)$. 
With the notation of  Theorem \ref{thm:toeplitz2} it follows that 
 $$u=B_1(\lambda)A_1(\lambda)u\in B_1(\lambda)(H^{s}(M,F,P))\subset 
     H^{s+\mu}(M,F,P).$$
Hence the closure and the maximal closed extension coincide and have domain $H^{s+\mu}(M,F,P)$. 
The stated norm estimate follows by writing $z=\tau^\mu e^{i\theta}$ and using that then 
 $$\|(z-A_P)^{-1}\|_{\scrL(H^s(M,F,P))}=\|PB(\lambda)R(\lambda)P\|_{\scrL(H^s(M,F,P))}\le 
     C |\tau|^{-\mu}=C|z|^{-1}$$
for $|z|$ large enough (the estimate is valid, since $R(\lambda)\in L^{-\mu}_\cl(\Lambda;M,F,F))$. 
\end{proof}

Let us mention the following alternative way to prove the previous theorem: With an elliptic operator 
$B\in L^\mu_\cl(M,F,F)$ having scalar principal symbol $-|\xi|^\mu$, let 
 $$C=PAP+(1-P)B(1-P)\;\in\; L^\mu_\cl(M,F,F).$$ 
Then $\tau^\mu-C=P(\tau^\mu-A)P+(1-P)(\tau^\mu-B)(1-P)$
and $c^{(\mu)}:\pi_1^* F\to\pi_1^* F$, the homogeneous principal symbol of $C$, 
does  not have spectrum in $\Lambda^\wedge$ if and only if this is true for \eqref{eq:spec}. 
We can apply the calculus of \cite{Grub} to obtain a parametrix $D(\tau)$ which 
coincides with $(\tau^\mu-C)^{-1}$ for large $\tau$. Then $PD(\tau)P$ is the inverse of 
$P(\tau^\mu-A)P$ for large $\tau$. 

\textbf{Acknowledgement:} The author thanks Professors Savin and Sternin for bringing their work \cite{SaSt} 
to his attention. 

\begin{small}
\bibliographystyle{amsalpha}

\end{small}

\end{document}